%% file: trpideals_arXiv_20161203.tex
\documentclass[11pt]{amsart}

\input LaTeXdefs.tex

\usepackage{mathrsfs}
\usepackage[colorlinks, citecolor=blue, linkcolor=red, pdfstartview=FitB]{hyperref}
\usepackage{amssymb, amsmath}
\usepackage{bm}

\renewcommand{\phi}{\varphi}
\renewcommand{\rho}{\varrho}

\begin{document}

\title{Compact ideals and rigidity of representations for amenable operator algebras}

\thanks{The first author was partially supported by an FQRNT postdoctoral fellowship and an NSERC Discovery Grant.  The second author was partially supported by an NSERC Discovery Grant.}
\author
	[R. Clou\^{a}tre]{{Rapha\"el Clou\^{a}tre}}
\address
	{Department of  Mathematics\\
	University of Manitoba\\
	Winnipeg, Manitoba \\
	Canada  \ \ \ R3T 2N2}
\email{raphael.clouatre@umanitoba.ca}
\author
	[L.W. Marcoux]{{Laurent W.~Marcoux}}
\address
	{Department of Pure Mathematics\\
	University of Waterloo\\
	Waterloo, Ontario \\
	Canada  \ \ \ N2L 3G1}
\email{LWMarcoux@uwaterloo.ca}

\begin{abstract}
We examine rigidity phenomena for representations of amenable operator algebras which have an ideal of compact operators. We establish that a generalized version of Kadison's conjecture on completely bounded homomorphisms holds for the algebra if the associated quotient is abelian. We also prove that injective completely bounded representations of the algebra are similar to complete isometries. The main motivating example for these investigations is the recent construction of Choi, Farah and Ozawa of an amenable operator algebra that is not similar to a $C^*$-algebra, and we show how it fits into our framework. All of our results hold in the presence of the total reduction property, a property weaker than amenability.
\end{abstract}

\maketitle
\markboth{\textsc{R. Clou\^atre and L.W. Marcoux}}{\textsc{Rigidity of representations for amenable operator algebras}}

%=====================================================
%==========    SECTION 1
%=====================================================

\section{Introduction} \label{secintro}

Let $\cA$ be a Banach algebra.   By a \emph{representation} of $\cA$, we shall mean a continuous algebra homomorphism of $\cA$ into $B(\cH)$, the algebra of bounded linear operators on a Hilbert space $\cH$.  Should $\cA$ be an involutive Banach algebra, we do not ask that the representation preserve adjoints -- i.e. we do not require it to be a ${}^*$-representation. 

In \cite{Kad1955}, R.V. Kadison asked whether representations of $C^*$-algebras are necessarily similar to ${}^*$-representations.  Significant progress towards a solution to this problem has been achieved over the years. Of particular interest to us will be the result of  E.~Christensen~\cite{Chr1981}, who showed that amenable $C^*$-algebras enjoy this property. Despite this, to this day the general question remains unanswered. U. Haagerup shed new light on the problem in \cite{Haa1983s}, establishing that a representation of a $C^*$-algebra is similar to a ${}^*$-representation if and only if it is \emph{completely} bounded.  V.~Paulsen~\cite{Pau1984} generalised Haagerup's work by proving that a representation of an operator algebra is completely bounded if and only if it is similar to a completely contractive representation. Based upon these results, G.~Pisier~\cite{Pis2001} formulated a more general version of Kadison's similarity problem which makes sense in a non self-adjoint context.

\bigskip

\noindent{\textbf{The generalised similarity problem.}} Let $\cA$ be an operator algebra and let $\theta:\cA\to B(\cH)$ be a representation. Is $\theta$ necessarily completely bounded?
\bigskip

We say that the algebra $\cA$ has the \emph{SP property} if that problem has an affirmative answer.  It is worth observing that the SP property implies the existence of an increasing function $\varphi: [0, \infty) \to [0, \infty)$ such that for any representation $\theta$  we have the inequality  $\| \theta \|_{cb} \le \varphi( \| \theta \|)$.   In fact, Pisier~\cite{Pis1999} introduced his notion of degree to capture the finer properties of the function $\phi$.

The disc algebra, consisting of all those continuous functions on the closed complex unit disc that are holomorphic on the interior, is an example of an operator algebra which fails to have the SP property \cite{Pis1997}.  In fact, this result is Pisier's celebrated counterexample to the Halmos Problem~\cite{Hal1970}, which asked whether every polynomially bounded operator on a Hilbert space is similar to a contraction. In the positive direction, it is clear that the SP property is preserved under similarity, and so we infer from Christensen's theorem cited above~\cite{Chr1981} that any operator algebra which is similar to an amenable $C^*$-algebra must have the SP property. It is a natural impulse, then, to wonder whether every amenable operator algebra has the SP property, or better yet, to try to solve the generalised similarity problem for amenable operator algebras by showing that every such algebra is similar to a $C^*$-algebra. There are in fact a number of classes of amenable operator algebras which have been shown to be similar to $C^*$-algebras, including amenable algebras of compact operators and abelian amenable algebras (\cite{Sei1977},\cite{Gif1997},\cite{Cho2013},\cite{MP2015},\cite{CM2016rfd}).

However, it has recently been made manifest that this desirable similarity property does not hold in general. Indeed, an example of an amenable operator algebra which fails to be similar to a $C^*$-algebra was constructed in~\cite{CFO2014} by Y. Choi, I. Farah and N. Ozawa.  It follows that a positive solution to the SP problem for amenable operator algebras cannot be achieved through the aforementioned approach using similarities and Christensen's theorem. Interestingly, as we shall see below, their algebra \emph{does} have the SP property, and establishing and generalising this property was one of the driving motivations behind this work.  

The SP property for an operator algebra may be interpreted as a rigidity phenomenon for representations. The second main theme of this paper is another rigidity property related to injective completely bounded representations. We investigate when such maps are necessarily similar to  completely isometric representations. Part of our motivation stems from the fact that this latter class of maps is very rigid, as they lift to ${}^*$-isomorphisms of the associated $C^*$-envelopes; this is sometimes referred to as Arveson's implementation theorem~\cite{Arv1972}. The question of determining when a representation is similar to a complete isometry was considered by the first author in~\cite{Clo2015}, prompted by classification problems for certain Hilbert space contractions, and therein it was shown to have an affirmative answer in various contexts. We will show here that this question also has an affirmative answer for the operator algebra constructed by Choi, Farah and Ozawa.

In fact, we will be working with a class of operator algebras that is strictly larger than the class of  amenable ones. Its defining characteristic is the so-called \emph{total reduction property}. This notion was introduced by J.A. Gifford in \cite{Gif1997}, where he showed amongst many other things that the class of $C^*$-algebras which enjoy the SP property coincides with the class of $C^*$-algebras possessing the total reduction property.   It is worth mentioning that this characterization does not transfer over to the non self-adjoint world. Indeed, the algebra of upper-triangular $2 \times 2$ complex matrices constitutes an example of an operator algebra which does not have the total reduction property, but which nonetheless possesses the SP property \cite[page 54]{Gif1997}.

Next, we outline the organization of the paper and our main results. Section~\ref{secPr} introduces background material. In Section~\ref{secSP}, we analyze ideals of operator algebras with the total reduction property from the perspective of the SP property, and prove the following result (Theorem~\ref{thm3.08}), parts of which are already known.

\begin{thm}\label{thmA}
Let $\cA$ be an operator algebra with the total reduction property, and let $\cJ \subset \cA$ be a closed two-sided ideal of $\cA$.   The following statements are equivalent.
\begin{enumerate}

\item[\rm{(i)}] $\cA$ has the SP property. 

\item[\rm{(ii)}] Both $\cJ$ and $\cA/\cJ$ have the SP property.

\end{enumerate}
\end{thm}

The following interesting consequence (Corollary \ref{cor3.06}) is one of our main results.

%------------------------------------------------------------
\begin{cor}
Let $\cA \subset B(\cH)$ be an operator algebra with the total reduction property, and let $\cJ\subset \cA$ be a closed two-sided ideal of $\cA$ which consists of compact operators on $\cH$.   Assume that $\cA/ \cJ$ is abelian.  Then, $\cA$ has the  SP property.
\end{cor}
%------------------------------------------------------------

In Section \ref{seccis}, we turn to the question of determining which injective, completely bounded representations are similar to complete isometries and we establish our other main result (Corollary~\ref{cor5.06}).
%------------------------------------------------------------
\begin{thm} \label{thmB}
Let $\cA \subset B(\cH)$ be an operator algebra with the total reduction property, and let $\cJ\subset \cA$ be a closed two-sided ideal of $\cA$ which acts non-degenerately on $\cH$.   Assume that $\cJ$ consists of compact operators on $\cH$. Let $\theta: \cA \to B(\hilb_\theta)$ be an injective completely bounded representation.  
Then, there exist invertible operators $X \in B(\hilb)$ and $Y \in B(\hilb_\theta)$ such that the map
\[
X^{-1} a X \mapsto Y^{-1} \theta(a) Y, \ \ \ a \in \cA \]
is completely isometric.
\end{thm}
%------------------------------------------------------------

Finally, in Section \ref{secCFO}, we recall the basic features of the operator algebra constructed by Choi, Farah and Ozawa. We then apply the general results proved in the previous sections and obtain the following (Theorems \ref{thmCFOSP} and \ref{thmCFOcis}).

%------------------------------------------------------------

\begin{thm} \label{thmC}
Let $\cA$ be the operator algebra of Choi, Farah and Ozawa. Then $\cA$ has the  SP property. Moreover, every injective representation of $\cA$ is similar to a completely isometric representation.
\end{thm}

%------------------------------------------------------------

%=====================================================
%==========    SECTION 2
%=====================================================

\section{Preliminaries}\label{secPr}

\subsection{Amenable operator algebras and completely bounded maps}

An \emph{operator algebra} is a subalgebra of $B(\cH)$ which is closed in the norm topology.  
If $\cA\subset B(\cH)$ is an operator algebra, then for each integer $n \ge 1$ the algebra $\bbM_n(\cA)$ of $n\times n$ matrices over $\cA$ may be endowed with a norm obtained by viewing it as a subalgebra of the $C^*$-algebra $B(\hilb^{(n)})$. If $\cB$ is another operator algebra and $\phi: \cA \to \cB$ is a linear map, then for each integer $n\geq 1$ we obtain a bounded linear map 
\[
\phi^{(n)}: \bbM_n(\cA) \to \bbM_n(\cB)
\]
by setting 
\[
\phi^{(n)} ([a_{i,j}] )= [ \phi(a_{i,j})]
\]
for all $[a_{i,j}] \in \bbM_n(\cA)$.   We say that $\phi$ is \emph{completely bounded} if the quantity
\[
\norm \phi \norm_{cb} = \sup_{n \ge 1} \norm \phi^{(n)} \norm
\]
 is finite.  We say that $\phi$ is \emph{completely contractive} (respectively, \emph{completely isometric}), if each map $\phi^{(n)}$ is contractive (respectively, isometric). An excellent standard reference about these topics is \cite{Pau2002}. We require the following well-known fact, which is \cite[Theorem 3.1]{Pau1984pb} (see also \cite{Pau1984} for a quantitative version).
 
 \begin{thm}\label{thmPaulsen}
 Let $\cA$ be an operator algebra and let $\theta:\cA\to B(\cH)$ be a completely bounded homomorphism. Then, there is an invertible operator $Y\in B(\cH)$ with the property that the map
 \[
 a\mapsto Y\theta(a)Y^{-1}, \quad a\in \cA
 \]
 is completely contractive.
 \end{thm}

Although we will not be dealing with amenability directly, we still recall its definition here, since amenable operator algebras are important examples of the more general algebras to which our results apply.

Let $X$ be a Banach $\cA$-bimodule.   A \emph{derivation} $\delta: \cA \to X$ is a continuous linear map which satisfies 
\[
\delta (a b) = \delta(a) b + a \delta (b)
\]
for all $a, b \in \cA$.    The derivation is \emph{inner} if there exists a fixed element $z \in X$ so that 
\[
\delta (a) = a z - z a
\]
for all $a \in \cA$. Observe that the dual space $X^*$ becomes a \emph{dual Banach $\cA$-bimodule} under the \emph{dual actions} $[a \cdot x^*](x) = x^*(x a)$ and $[x^* \cdot a](x) = x^*(a x)$ for all $a \in \cA$, $x^* \in X^*$ and $x \in X$.   We say that $\cA$ is \emph{amenable} if every derivation of $\cA$ into a dual Banach $\cA$-bimodule (equipped with the dual actions) is inner.

%=================================
\subsection{The total reduction property}
 We now define the notion which is central to this paper, and which was originally introduced by J.A.~Gifford in his PhD thesis \cite{Gif1997}.  
 
 An operator algebra $\cA$ is said to have the \emph{total reduction property} if whenever $\theta: \cA \to B(\cH)$ is a representation and $M \subset \hilb$ is a closed $\theta(\cA)$-invariant subspace, there exists another closed $\theta(\cA)$-invariant subspace $N$ which is a topological complement of $M$, in the sense that $\hilb = M + N$ and $M \cap N  = \{ 0\}$. Equivalently,  $\cA$ has the total reduction property if any closed $\theta(\cA)$-invariant subspace is the range of some bounded idempotent that commutes with $\theta(\cA)$.  Clearly, the total reduction property is preserved by similarity.

Any amenable operator algebra has the total reduction property \cite[Proposition 2.3.2]{Gif1997}. However, the total reduction property is strictly weaker than amenability. Indeed, if $\cH$ is an infinite dimensional Hilbert space, then $B(\cH)$ has the total reduction property \cite[Corollary~2.4.7]{Gif1997} but it is not amenable since it is not nuclear \cite{Sza1981}, \cite{Con1978}.

Before proceeding, we gather here several useful facts about operator algebras with the total reduction property that we require frequently in the sequel. First, we examine ideals.

\begin{thm}\label{thmideal}
Let $\cA$ be an operator algebra with the total reduction property and let $\cJ\subset \cA$ be a closed two-sided ideal. Then, $\cJ$ has the total reduction property. Moreover, $\cJ$ admits a bounded approximate identity: there exists a bounded net $(e_\lambda)_\lambda$ in $\cJ$ such that
\[
\lim_\lambda \|ae_\lambda-a\|=\lim_\lambda \|e_\lambda a-a\|=0
\]
for every $a\in \cJ$.
\end{thm}
\begin{proof}
This follows from \cite[Propositions 3.2.7 and 3.3.3]{Gif1997}. 
\end{proof}

%=====================================================

As mentioned in the introduction, one of the most interesting features of the total reduction property is that it characterizes the SP property for $C^*$-algebras \cite[Corollary 2.4.5]{Gif1997}.

%=====================================================

\begin{thm}\label{thmGifKSP}
Let $\fA$ be a $C^*$-algebra. Then, $\fA$ has the SP property if and only if it has the total reduction property. 
%In that case, given a representation $\theta:\fA\to B(\cH)$ we must satisfy
%\[
%\|\theta\|_{cb}\leq 128 \kappa_{\fA}(\|\theta\|)^2.
%\]
\end{thm}

%=====================================================

Although a general amenable operator algebra is not necessarily similar to a $C^*$-algebra~\cite{CFO2014}, the statement is valid for subalgebras of compact operators and for abelian algebras. In fact, amenability can even be replaced by the total reduction property.

%=====================================================

\begin{thm}\label{thmsimC*}
Let $\cA$ be an operator algebra with the total reduction property. Then, $\cA$ is similar to a $C^*$-algebra if one of the following conditions holds:
\begin{enumerate}

\item[\rm(1)] $\cA$ consists of compact operators,

\item[\rm(2)] $\cA$ is abelian.
\end{enumerate}
\end{thm}
\begin{proof}
Combine \cite[Theorem 4.3.13]{Gif1997} and \cite[Theorem 2.10]{MP2015}
\end{proof}

%=====================================================

We close this preliminary section by addressing a technical point. Recall that a representation $\theta:\cA\to B(\cH)$  is said to \emph{act non-degenerately} if $\ol{\theta(\cA) \cH}=\cH$.  The algebra $\cA$ is said to act non-degenerately if the identity representation does. While representations of operator algebras with the total reduction property do not necessarily act non-degenerately, they nearly do.

%=====================================================

\begin{lem}\label{lemnondeg}
Let $\cA$ be an operator algebra with the total reduction property and let $\theta:~\cA\to~B(\cH)$ be a representation. Then, there exist a subspace $\cH_0\subset \cH$, a non-degenerately acting representation $\theta_0:\cA\to B(\cH_0)$ and an invertible operator $X\in B(\cH)$  such that 
\[
X\theta(a) X^{-1}=\theta_0(a)\oplus 0, \quad a\in \cA
\]
according to the orthogonal decomposition $\cH=\cH_0\oplus \cH_0^\perp$.
\end{lem}
\begin{proof}
The algebra $\ol{\theta(\cA)}$ is known to have the total reduction property by \cite[Proposition 3.3.1]{Gif1997}. Then, by \cite[Lemma 2.5]{CM2016rfd}, there is a invertible operator $X\in B(\cH)$ 
along with a subspace $\cH_0\subset \cH$ which is invariant for $X\theta(\cA)X^{-1}$ such that if we set 
\[
\cB_0=X\ol{\theta(\cA)}X^{-1}|_{\cH_0}\subset B(\cH_0)
\]
then
\[
X\ol{\theta(\cA)} X^{-1}=\cB_0\oplus \{0\}
\]
according to the orthogonal decomposition $\cH=\cH_0\oplus \cH_0^\perp$. Furthermore, $\cB_0$ 
is a non-degenerately acting algebra with the total reduction property.
The proof is completed by defining  $\theta_0:\cA\to B(\cH_0)$ as 
\[
\theta_0(a)=X\theta(a)X^{-1}|_{\cH_0}, \quad a\in \cA.
\]
\end{proof}

%=====================================================

%=====================================================
%==========    SECTION 3
%=====================================================

\section{The SP property for ideals and quotients} \label{secSP}

The aim of this section is to unravel the relationship between the SP property for an algebra, and the SP property for an ideal of the algebra and the associated quotient. We then exploit this relationship to prove one of our main results. The basic tool is the following.

\begin{lem}\label{lemtrprep}
Let $\cA$ be an operator algebra with the total reduction property and let $\cJ\subset \cA$ be a closed two-sided ideal of $\cA$. Let $\theta:\cA\to B(\cH)$ be a representation. Then, there exist a closed subspace $\cH_0\subset \cH$ and two representations 
\[
\rho:\cA\to B(\cH_0), \quad  \tau:\cA\to B(\cH_0^\perp) 
\]
with the following properties:
\begin{enumerate}

\item[\rm{(a)}] the representation $\rho|_{\cJ}:\cJ\to B(\cH_0)$ acts non-degenerately,

\item[\rm{(b)}] $\tau(\cJ)=0$,

\item[\rm{(c)}] there exists an invertible operator $X\in B(\cH)$ such that
\[
X\theta(a)X^{-1}=\rho(a)\oplus \tau(a), \quad a\in \cA.
\]
\end{enumerate} 
\end{lem}
\begin{proof}
Consider the subspace $\hilb_0 = \ol{\theta(\cJ) \hilb}$ which is seen to be invariant for $\theta(\cA)$.   Hence, relative to the orthogonal decomposition $\hilb = \hilb_0 \oplus \hilb_0^\perp$, we find
\[
\theta(a) = \begin{bmatrix} \rho(a) & \delta(a) \\ 0 & \tau(a) \end{bmatrix} \]
where 
\[
\rho(a)=\theta(a)|_{\cH_0}, \quad \delta(a)=P_{\cH_0}\theta(a)|_{\cH^\perp_0}, \quad \tau(a)=P_{\cH_0^\perp}\theta(a)|_{\cH^\perp_0}
\]
for every $a\in \cA$.
We note that $\rho$ and $\tau$ are both representations of $\cA$.  It is clear that $\tau(\cJ)=0$ by definition of $\cH_0$. Since $\cA$ has the total reduction property, the ideal $\cJ$ admits a bounded approximate identity $(e_\lambda)_\lambda$ by Theorem \ref{thmideal}.  For $\xi\in \cH$ and $j\in \cJ$ we see that
\[
\theta(j)\xi=\lim_\lambda \theta(j)\theta(e_\lambda)\xi\in \ol{\theta(\cJ)\cH_0},
\]
which shows that $\ol{\theta(\cJ)\cH_0}=\cH_0$.  Thus, $\rho|_\cJ$ acts non-degenerately.

Finally, because $\cA$ has the total reduction property there is an idempotent $E \in B(\hilb)$ that commutes with $\theta(\cA)$ such that $E\cH=\cH_0$. We may write 
\[
E = \begin{bmatrix} I & E_2 \\ 0 & 0 \end{bmatrix}
\]
for some bounded linear operator $E_2: \cH_0^\perp \to \cH_0$. Consider now the invertible operator
\[
X = \begin{bmatrix} I & -E_2 \\ 0 & I \end{bmatrix}\in B(\cH).
\]
Using that $E$ commutes with $\theta(\cA)$, a straightforward calculation establishes that
\[
X^{-1} \theta(a) X = \begin{bmatrix} \rho(a) & 0 \\ 0 & \tau(a) \end{bmatrix}
 \]
 for every $a\in \cA$. 
\end{proof}

Next, we record a standard extension result for representations, modelled closely on the corresponding result for ${}^*$-representations of $C^*$-algebras.

\begin{lem} \label{lemext}
Let $\cA$ be an operator algebra and $\cJ \subset \cA$ be a closed two-sided ideal of $\cA$ which admits a bounded approximate identity $(e_\lambda)_\lambda$.  Let $\theta:\cJ\to B(\cH)$ be a  non-degenerately acting representation.  Then, there is a unique representation $\widehat{\theta}:\cA\to B(\cH)$ which extends $\theta$ and which satisfies
\[
\|\widehat{\theta}^{(n)}\|\leq \left(\sup_\lambda\|e_\lambda\|\right)\|\theta^{(n)}\|
\]
for every $n\in \bbN$.   In particular, if $\theta$ is completely bounded, then so is its extension $\widehat{\theta}$.
\end{lem}

\begin{proof}
The existence and uniqueness of the representation $\widehat{\theta}$ follows using a standard argument, based on the fact that $\cH=\ol{\theta(\cJ)\cH}$ and that $(e_\lambda)_\lambda$ is a bounded approximate identity (see for instance \cite[page 14-15]{Arv1976}). In particular, we have that
\[
\widehat{\theta}(a)\theta(j)\xi=\theta(aj)\xi,
\]
for every $a\in \cA, j\in \cJ, \xi\in \cH$ .

Next,  fix $n\in \bbN$ and put $M=\sup_\lambda \|e_\lambda\|$. For each $\lambda$,  define
\[
e_\lambda^{(n)}=e_\lambda\oplus e_\lambda\oplus \ldots \oplus e_\lambda\in \bbM_n(\cJ).
\]
It is easily verified that $(e_\lambda^{(n)})_\lambda$ is a bounded approximate identity for $\bbM_n(\cJ)$ with 
\[
\sup_\lambda \|e^{(n)}_\lambda\|=M.
\]
Moreover, since $\cH=\ol{\theta(\cJ)\cH}$, we see that $(\theta^{(n)}(e^{(n)}_\lambda))_\lambda$ converges strongly to the identity operator on $\cH^{(n)}$. Let $A \in \bbM_n(\cA)$.   For $\xi \in \hilb^{(n)}$, we get 
\begin{align*}
\norm {\widehat{\theta}}^{(n)}(A) \xi \norm
	&= \lim_\lambda \norm {\widehat{\theta}}^{(n)} (A) \theta^{(n)}(e^{(n)}_\lambda)\xi \norm \\
	&= \lim_\lambda \norm \theta^{(n)} (Ae^{(n)}_\lambda) \xi \norm \\
	&\le M\, \norm \theta^{(n)} \norm  \, \|A\| \, \norm \xi \norm, 
\end{align*}
so that $\norm \widehat{\theta} ^{(n)}\norm \le M \norm \theta^{(n)} \norm$. The last statement follows immediately from this.
\end{proof}

Using these two results, we show that for operator algebras with the total reduction property, the SP property can be completely understood through ideals and quotients. 
\begin{thm} \label{thm3.08}
Let $\cA$ be an operator algebra with the total reduction property.  Let $\cJ \subset \cA$ be a closed two-sided ideal of $\cA$.   The following statements are equivalent.
\begin{enumerate}
	\item[(i)]
	$\cA$ has the SP property.
	\item[(ii)]
	 Both $\cJ$ and $\cA/\cJ$ have the SP property.
\end{enumerate}	 
\end{thm}

\begin{proof}
First, note that since $\cA$ has the total reduction property, the ideal $\cJ$ admits a bounded approximate identity by Theorem \ref{thmideal}.  Moreover, throughout the proof we let $\pi:\cA\to \cA/\cJ$ denote the quotient map. We now proceed to show the equivalence of (i) and (ii).

Assume first that (i) holds. Let $\theta:\cA/\cJ\to B(\cH)$ be a representation. Then, we have
\[
\|\theta^{(n)}\circ \pi^{(n)}\|=\|\theta^{(n)}\|
\]
for every $n\in \bbN$. Thus, we immediately see that $\cA/\cJ$ has the SP property if $\cA$ has it.   This has nothing to do with the total reduction property.   We now turn to the ideal $\cJ$.  By Lemma~\ref{lemnondeg}, to show that $\cJ$ also has the SP property, it suffices to show that all non-degenerately acting representations of $\cJ$ are completely bounded. Accordingly, let $\theta: \cJ \to B(\cH)$ be a  representation which acts non-degenerately.   By Lemma \ref{lemext}, $\theta$ can be uniquely extended to a representation $\widehat{\theta}: \cA \to B(\hilb)$. Since $\cA$ has the SP property, we see that $\widehat{\theta}$ is completely bounded. Clearly, we have that $\|\theta\|_{cb}\leq \|\widehat{\theta}\|_{cb}$ so that $\theta$ is completely bounded as well. We conclude that $\cJ$ has the SP property and thus that (ii) is verified.

Next, we assume that (ii) holds and show that (i) follows. Let $\theta: \cA \to B(\hilb)$ be a representation.  By Lemma \ref{lemtrprep}, there exist a closed subspace $\cH_0\subset \cH$, two representations 
\[
\rho:\cA\to B(\cH_0), \quad  \tau:\cA\to B(\cH_0^\perp) 
\]
and an invertible operator $X\in B(\cH)$ such that
\[
X\theta(a)X^{-1}=\rho(a)\oplus \tau(a), \quad a\in \cA.
\]
Moreover, $\tau(\cJ)=0$ and  the representation $\rho|_{\cJ}:\cJ\to B(\cH_0)$ acts non-degenerately.

Observe that $\rho|_{\cJ}$ must be completely bounded since $\cJ$ is assumed to have the SP property.   In addition, by Lemma~\ref{lemext} we see that $\rho|_\cJ$ admits a unique extension to a representation $\widehat{\rho|_{\cJ}}$ of $\cA$, which must be completely bounded. By uniqueness of the extension, we conclude that $\rho = \widehat{\rho|_{\cJ}}$ and thus that $\rho$ is completely bounded.   Next, since $\tau(\cJ)=0$ there exists a representation $\widehat{\tau}:\cA/\cJ\to B(\cH_0^\perp)$ such that $\tau=\widehat{\tau}\circ \pi$.  Since $\cA/\cJ$ is assumed to have the SP property, we infer that $\widehat{\tau}$ is completely bounded. But $\pi$ is completely contractive, whence we deduce that $\tau$ is completely bounded. Therefore, $\theta$ is completely bounded, and (i) follows.

\end{proof}

With the help of his powerful ``length" machinery~\cite{Pis1999}, Pisier has established parts of the previous theorem~\cite[Exercise 27.1]{Pis2003}. For arbitrary operator algebras, he shows that (ii) implies (i), while (i) implies that $\cA/\cJ$ has the SP property.   Furthermore, if $\cA$ is a $C^*$-algebra or if $\cA \simeq \cJ \times \cA/\cJ$, then he recovers that (i) implies (ii).    Exploiting our assumption that $\cA$ has the total reduction property, we provided above a different approach for showing the equivalence of (i) and (ii). In our setting, it appears to be new that the SP property of $\cA$ is inherited by the ideal $\cJ$.

%-----------------------------------------------------------------------

Finally, we arrive at the main result of this section.  We will see in Section \ref{secCFO} that it can be applied to the construction of Choi, Farah and Ozawa.

\begin{cor} \label{cor3.06}
Let $\cA \subset B(\cH)$ be an operator algebra with the total reduction property, and let $\cJ\subset \cA$ be a closed two-sided ideal of $\cA$ which consists of compact operators on $\cH$.   Assume that $\cA/ \cJ$ is abelian.  Then, $\cA$ has the  SP property.
\end{cor}

\begin{proof}
We note that $\cJ$ has the total reduction property by Theorem \ref{thmideal}, and so does $\cA/\cJ$ by \cite[Proposition 3.3.1]{Gif1997}.  Theorem \ref{thmsimC*} then implies that both $\cJ$ and $\cA/\cJ$ are similar to $C^*$-algebras.  By virtue of Theorem \ref{thmGifKSP}, we infer that $\cJ$ and $\cA/\cJ$ have the SP property, so that $\cA$ has the SP property in view of Theorem \ref{thm3.08}.
\end{proof}

%=====================================================
%==========    SECTION 4
%=====================================================

\section{Similarity to complete isometries} \label{seccis}

In this section, we consider a different rigidity phenomenon for representations of operator algebras with the total reduction property. Recall that by Theorem \ref{thmPaulsen}, any completely bounded representation of an operator algebra is similar to a completely contractive one. A natural question then is whether a completely bounded representation with a completely bounded inverse is necessarily similar to a complete isometry. In other words, we aim to obtain a ``symmetric" version of Theorem \ref{thmPaulsen}. We will show that this can be achieved under the assumption that the algebra contains a large enough ideal of compact operators.

First, we show that if an operator algebra (not necessarily self-adjoint) admits a self-adjoint ideal, then its completely contractive representations split according to this ideal, in the same way that ${}^*$-representations do \cite[page 15]{Arv1976}. The proof is an adaptation of the standard one, but we include it for the sake of completeness.
 
%=====================================================
 
\begin{lem}\label{lem5.02}
Let $\cA$ be an operator algebra. Assume that there is a closed two-sided ideal $\cJ$ of  $C^*(\cA)$ with the property that $\cJ\subset \cA$. Let $\theta:\cA\to B(\hilb)$ be a completely contractive representation. Then, there exists an orthogonal decomposition $\cH=\cH_1\oplus \cH_2$, a\linebreak  ${}^*$-representation $\sigma:C^*(\cA)\to B(\hilb_1)$ and a completely contractive representation\linebreak $\tau:~\cA/\cJ\to~B(\hilb_2)$ with the property that
$$
\theta(a)=\sigma(a) \oplus \tau(a+\cJ)
$$
for every $a\in \cA$.
\end{lem}

\begin{proof}
The ideal $\cJ$ is a $C^*$-algebra.
Since $\theta$ is completely contractive, we see that $\theta|_\cJ$ is a ${}^*$-representation.
Set $\hilb_1=\ol{\theta(\cJ)\hilb}$. It is clear that $\cH_1$ is invariant under $\theta(\cA)$. We claim that in fact $\cH_1$ is reducing for $\theta(\cA)$. Indeed, using that $\cJ$ is self-adjoint and that $\theta|_{\cJ}$ is a ${}^*$-representation, given $a\in \cA$ and $j\in \cJ$ we have
\begin{align*}
\theta(a)^*\theta(j)&=(\theta(j)^*\theta(a))^*=(\theta(j^*)\theta(a))^*\\
&=\theta(j^*a)^*=\theta(a^*j)
\end{align*}
and thus $\theta(a)^*\theta(\cJ)\subset \theta(\cJ)$. In particular, we see that $\theta(a)^*\cH_1\subset \cH_1$ so that  $\hilb_1$ is a reducing subspace for $\theta(\cA)$.
Hence, the map
$$
\sigma_0:\cJ\to B(\hilb_1)
$$
defined by  
\[
\sigma_0(j)=\theta(j)|_{\hilb_1}, \quad j\in \cJ
\]
is a ${}^*$-representation. Since $\cJ$ is a $C^*$-algebra, it admits a bounded approximate identity, and arguing as in the proof of Lemma~\ref{lemtrprep} we see that $\sigma_0$ is a non-degenerately acting ${}^*$-representation. Thus, it may be extended uniquely to a ${}^*$-representation  $\sigma:C^*(\cA)\to B(\hilb_1)$ (see \cite[page 14-15]{Arv1976}). Next, 
we note that for $\xi\in \hilb_1, j\in \cJ$ and $a\in \cA$ we have
\begin{align*}
\sigma(a)\sigma_0(j)\xi&=\sigma(a)\sigma(j)\xi=\sigma(aj)\xi=\sigma_0(aj)\xi\\
&=\theta(aj)\xi=\theta(a)\theta(j)\xi=\theta(a)\sigma_0(j)\xi
\end{align*}
and thus $\sigma(a)=\theta(a)|_{\hilb_1}$ since $\sigma_0$ acts non-degenerately.
Put $\hilb_2=\hilb\ominus \hilb_1$. Then
$$
\hilb_2=(\theta(\cJ)\hilb)^\perp=\cap_{j\in \cJ}\ker \theta(j^*)=\cap_{j\in \cJ}\ker \theta(j)
$$
since $\cJ$ is self-adjoint. Therefore, $\theta(\cJ)\hilb_2=\{0\}$ and we may define a completely contractive representation
$$
\tau:\cA/\cJ\to B(\hilb_2)
$$
via
$$
\tau(a+\cJ)=\theta(a)|_{\hilb_2}, \quad a\in \cA.
$$
Finally, since $\hilb_1$ is reducing for $\theta(\cA)$ we find
\begin{align*}
\theta(a)&=\theta(a)|_{\hilb_1}\oplus \theta(a)|_{\hilb_2}=\sigma(a)\oplus \tau (a+\cJ)
\end{align*}
for every $a\in \cA$.
\end{proof}

%=====================================================

Recall that if $\cA\subset B(\cH)$ is an operator algebra, then a subset $\cX\subset B(\cH)$  is \emph{essential for $\cA$} if, for any element $a\in \cA$, the condition
\[
ax=0= xa \quad \text{for all} \quad x\in \cX
\]
implies that $a=0$. It is easily verified that if $\cX$ acts non-degenerately on $\cH$, then it must be essential for $\cA$.

The following is the main technical tool of this section, which generalizes the fact that injective ${}^*$-representations of $C^*$-algebras are completely isometric.

%=====================================================

\begin{thm}\label{thm5.04}
Let $\cA$ be an operator algebra. Assume that there is a closed two-sided ideal $\cJ$ of  $C^*(\cA)$ with the property that $\cJ\subset \cA$. Assume also that $\cJ$ is essential for $C^*(\cA)$.  Then, any injective completely contractive representation of $\cA$ is completely isometric.

\end{thm}

\begin{proof}
Let $\theta:\cA\to B(\cH)$ be an injective completely contractive representation.
By Lemma~\ref{lem5.02} there exist a subspace $\hilb_1\subset\hilb$ which is reducing for $\theta(\cA)$, a ${}^*$-representation
$$
\sigma:C^*(\cA)\to B(\hilb_1)
$$
and a completely contractive representation
$$
\tau:\cA/\cJ\to B(\hilb_1^\perp)
$$
with the property that
$$
\theta(a)=\sigma(a)\oplus \tau(a+\cJ)
$$
for every $a\in \cA$.

Now,  we claim that $\sigma$ has trivial kernel. Indeed, assume that $\sigma(x)=0$ for some $x\in C^*(\cA)$. Then, for every $j\in \cJ$ we have $jx\in \cJ\subset \cA$, $xj\in \cJ\subset \cA$ and
$$
\theta(jx)=\sigma(jx)\oplus 0=0, \quad \theta(xj)=\sigma(xj)\oplus 0=0
$$
whence $jx=0=xj$, since $\theta$ is assumed to be injective. By assumption, $\cJ$ is essential for $C^*(\cA)$ so that $x=0$. We conclude that $\sigma$ is an injective ${}^*$-representation,  and thus it is completely isometric. Since $\tau$ is completely contractive, we conclude that $\theta$ must be completely isometric as well.
\end{proof}

%=====================================================

It is now a simple matter to adapt the previous result so that it applies more generally to injective representations that are merely completely bounded.

%=====================================================

\begin{cor}\label{cor5.05}
Let $\cA$ be an operator algebra. Assume that there is a closed two-sided ideal $\cJ$ of  $C^*(\cA)$ with the property that $\cJ\subset \cA$. Assume also that $\cJ$ is essential for $C^*(\cA)$. Let $\theta:\cA\to B(\hilb)$ be an injective completely bounded representation. Then, there exists an invertible operator $Y\in B(\hilb)$ such that the map
$$
a\mapsto Y\theta(a)Y^{-1}
$$
is completely isometric.
\end{cor}

\begin{proof}
Combine Theorem \ref{thmPaulsen} with Theorem \ref{thm5.04}.
\end{proof}

We remark that if $\cA$ is an operator algebra and $\cJ\subset \cA$ is a closed two-sided ideal of $\cA$ which happens to be self-adjoint, then in fact $\cJ$ is an ideal of $C^*(\cA)$. Indeed, given $a\in \cA$ and $j\in \cJ$ we see that $j^*\in \cJ$, whence $j^* a\in \cJ$ and
\[
a^* j=(j^* a)^*\in \cJ.
\]
Likewise, we find $ja^*\in \cJ$. We can now give an application of Corollary \ref{cor5.05} to operator algebras with the total reduction property.

%=====================================================

\begin{cor} \label{cor5.06}
Let $\cA \subset B(\cH)$ be an operator algebra with the total reduction property, and let $\cJ\subset \cA$ be a  non-degenerately acting, closed, two-sided ideal of $\cA$.   Assume that $\cJ$ consists of compact operators. Let $\theta: \cA \to B(\hilb_\theta)$ be an injective completely bounded representation.  
Then, there exist invertible operators $X \in B(\hilb)$ and $Y \in B(\hilb_\theta)$ such that the map
\[
X^{-1} a X \mapsto Y^{-1} \theta(a) Y, \ \ \ a \in \cA \]
is completely isometric.
\end{cor}

\begin{proof}
By Theorem~\ref{thmideal}, we see that $\cJ$ has the total reduction property. Thus, by Theorem~\ref{thmsimC*} there is an invertible operator $X\in B(\cH)$ such that $X^{-1}\cJ X$ is a non-degenerately acting $C^*$-algebra. Set $\cB = X^{-1} \cA X$, and define 
\[
\rho: \cB \to B(\hilb_\theta)
\]
as
\[
\rho (b) = \theta(XbX^{-1}), \quad b\in \cB.
\]
It is obvious that $\rho$ is an injective completely bounded representation of $\cB$. Since $X^{-1}\cJ X$ is a self-adjoint ideal in $\cB$, we conclude that $X^{-1} \cJ X$ is a closed two-sided ideal of $C^*(\cB)$, as noted above.  It is essential  for $C^*(\cB)$ because it acts non-degenerately. We may therefore invoke Corollary~\ref{cor5.05} to deduce the existence of an invertible operator $Y \in B(\hilb_\theta)$ such that the map
\[
b \mapsto Y^{-1} \rho(b) Y, \quad b\in \cB
\]
is completely isometric.  Equivalently, we see that
\[
X^{-1} a X \mapsto Y^{-1} \theta(a) Y , \quad a\in \cA\]
is completely isometric, and the proof is complete.
\end{proof}

%=====================================================

We end this section by strengthening the previous corollary, using the well-known fact that unital completely isometric linear isomorphisms between operator algebras lift to ${}^*$-isomorphisms of the so-called \emph{$C^*$-envelopes}. We will not be needing the precise definition of this object here, but the interested reader can consult \cite{Arv1972} and \cite{Pau2002} for details.   

%=====================================================

\begin{cor}\label{corC*env}
Let $\cA \subset B(\cH)$ be a unital operator algebra with the total reduction property, and suppose that $\cA$ contains all the compact operators on $\cH$.   Let $\theta: \cA \to B(\hilb_\theta)$ be a unital, injective, completely bounded representation.  
Then, there exist an invertible operator  $Y \in B(\hilb_\theta)$ and a unital, injective ${}^*$-representation  $\sigma: C^*(\cA)\to B(\cH_\theta)$ with the property that
\[
\sigma(a) = Y^{-1}\theta(a) Y, \quad a\in \cA.
\]
\end{cor}
\begin{proof}
The self-adjoint ideal $\cK(\cH)$ of compact operators on $\cH$ acts non-degenerately, so it is essential for $C^*(\cA)$. By Corollary \ref{cor5.05}, there exists an invertible operator  $Y \in B(\hilb_\theta)$ such that the map
\[
 a  \mapsto Y^{-1} \theta(a) Y, \ \ \ a \in \cA \]
is unital and completely isometric.   The existence of the ${}^*$-representation $\sigma$ then follows from \cite[Proposition 2.1.0, Theorem 2.1.1 and Theorem 0.3]{Arv1972} and\cite[Theorem 2.1.2]{Arv1969}.
\end{proof}

%=====================================================

%=====================================================
%==========    SECTION 5
%=====================================================

\section{The construction of Choi, Farah and Ozawa}\label{secCFO}

In this final section, we apply the main results of the previous sections to the example of Choi, Farah and Ozawa \cite{CFO2014}. We first recall the basic features of their construction.

Let $\cH=\bigoplus_{n\in \bbN} \bbC^2$ and define the following $C^*$-algebra
\[
\cL=\{ (a_n)_{n\in \bbN}: a_n\in \bbM_2 \text{ for every } n\in \bbN \text{ and } \sup_n\|a_n\|<\infty \}
\]
where $\bbM_2$ denotes the collection of $2\times 2$ complex matrices. Then, we see that $\cL\subset B(\cH)$. Next, we define 
\[
\cJ=\{ (a_n)_{n\in \bbN}\in \cL: \lim_{n\to \infty} \|a_n\|=0\},
\]
which is easily seen to be a closed two-sided ideal of $\cL$. Moreover, it is clear that $\cJ$ consists of compact operators on $\cH$, and that $\cJ$ acts non-degenerately. 

Let $\pi:\cL\to \cL/\cJ$ be the quotient map. In \cite{CFO2014}, it is shown that for a clever choice of abelian group $\Gamma$ and of uniformly bounded representation $\rho:\Gamma\to \pi(\cL)$, the operator algebra
\[
\cA_\Gamma=\pi^{-1}\left(\ol{ \text{span } \rho(\Gamma)}\right)
\]
is amenable, yet it is not similar to a $C^*$-algebra, thus giving a counterexample to a longstanding conjecture. For our purposes, the only properties of $\cA_\Gamma$ that will be needed are that it contains $\cJ$ and that
\[
\cA_\Gamma/\cJ\cong \ol{ \text{span } \rho(\Gamma)}
\]
is abelian, since $\Gamma$ is abelian.

%=====================================================

\begin{thm}\label{thmCFOSP}
The algebra $\cA_\Gamma$ has the  SP property.
\end{thm}
\begin{proof}
Since $\cA_\Gamma$ is amenable, it must have the total reduction property. The conclusion then follows directly from Corollary \ref{cor3.06}.
\end{proof}

%=====================================================

Combining this theorem with the results from Section \ref{seccis}, we obtain the following.

%=====================================================

\begin{thm}\label{thmCFOcis}
Let $\theta: \cA_\Gamma \to B(\cH)$ be an injective representation. Then, there exists an invertible operator $Y \in B(\hilb)$ such that the map 
\[
a \mapsto Y^{-1} \theta(a) Y , \quad a\in \cA_\Gamma\]
is completely isometric.
\end{thm}
\begin{proof}
By Theorem \ref{thmCFOSP}, we see that $\theta$ is completely bounded. Since $\cJ\subset \cA_\Gamma$ is a closed two-sided essential ideal of $\cL$, the conclusion follows from Corollary \ref{cor5.05}.
\end{proof}

%=====================================================

We finish by pointing out that in his thesis~\cite{Gif1997}, Gifford conjectured that every weak${}^*$-closed operator algebra with the total reduction property must be similar to a $C^*$-algebra.  More precisely, he conjectured this for weak${}^*$-closed operator algebras with a weaker property known as the \emph{complete reduction property}.
The algebra $\cA_\Gamma$ above supports Gifford's intuition: although $\cA_\Gamma$ is not similar to a $C^*$-algebra despite being amenable, its weak-${}^*$ closure coincides with the $C^*$-algebra $\cL$. 

In general, it is not known whether the weak${}^*$-closure of an operator algebra with the total reduction property also enjoys the total reduction property (although it must have the weaker complete reduction property \cite{Gif1997}).  In the case at hand however, this can be seen directly since $\cL$ can be realized as the tensor product $\ell_\infty(\bbN) \otimes \bbM_2$. Now, $\ell_\infty(\bbN)$ is an abelian $C^*$-algebra, hence it is nuclear and thus so is $\cL$. It follows that $\cL$ is amenable \cite{Haa1983} and therefore it has the total reduction property.

We refer the reader to the recent preprint~\cite{CM2016rfd} for a closer examination of the similarity to $C^*$-algebras of subalgebras of products of matrix algebras that have the total reduction property.

\bibliographystyle{plain}
\bibliography{2015_11papers}

\end{document}

%% file: LaTeXdefs.tex
\usepackage{latexsym}
\usepackage{amsfonts}

\newenvironment{thm}{\subsection{}{\textbf {Theorem.}}\em}{}

\newenvironment{cor}{\subsection{}{\textbf {Corollary.}}\em}{}
\newenvironment{lem}{\subsection{}{\textbf {Lemma.}}\em}{}
\newcommand\fA{\ensuremath{\mathfrak A}}

\newcommand\cA{\ensuremath{\mathcal A}}
\newcommand\cB{\ensuremath{\mathcal B}}

\newcommand\cH{\ensuremath{\mathcal H}}

\newcommand\cJ{\ensuremath{\mathcal J}}
\newcommand\cK{\ensuremath{\mathcal K}}
\newcommand\cL{\ensuremath{\mathcal L}}

\newcommand\cX{\ensuremath{\mathcal X}}

\newcommand\bbC{\ensuremath{\mathbb C}}

\newcommand\bbM{\ensuremath{\mathbb M}}
\newcommand\bbN{\ensuremath{\mathbb N}}

\newcommand\hilb{\ensuremath{\mathcal H}}

\newcommand\ol{\ensuremath{\overline}}

\newcommand\eop{{{\hfil \ensuremath \Box}}}
\newcommand\norm{\ensuremath {\Vert}}